\documentclass[a4paper,leqno,11pt]{amsart}

\usepackage{amsfonts,amssymb,stackrel,verbatim,amsmath,amsthm,latexsym,textcomp,amscd}
\usepackage{latexsym,amsfonts,amssymb,epsfig,verbatim}
\usepackage{amsmath,amsthm,amssymb,latexsym,graphics,textcomp}
\usepackage{mathtools}
\usepackage{paralist}
\usepackage{graphicx}
\usepackage{color}
\usepackage{url}
\usepackage{enumerate}
\usepackage[mathscr]{euscript}
\usepackage{tikz-cd}
\usetikzlibrary{shapes.geometric}
\usepackage{dsfont}
\usetikzlibrary{matrix}
\usepackage{hyperref}

\input xy
\xyoption{all}

\DeclareMathOperator{\Map}{Map}

\theoremstyle{definition}
\newtheorem{theorem}{Theorem}[section]
\newtheorem{prop}[theorem]{Proposition}
\newtheorem{lemma}[theorem]{Lemma}

\newtheorem{defn}[theorem]{Definition}

\newtheorem{rmk}[theorem]{Remark}

\newtheorem{exam}[theorem]{Example}

\newtheorem{thm}[theorem]{Theorem}
\newtheorem{notation}[theorem]{Notation}
\newtheorem{subsec}[theorem]{}

\theoremstyle{remark}
{\swapnumbers
   \newtheorem{ack}[theorem]{Acknowledgements}}

\newcommand{\R}{{\mathbb R}}
\newcommand{\Q}{{\mathbb Q}}

\newcommand{\Hom}{\mathrm{Hom}}

\newcommand{\Top}{\mathsf{Top}}
\newcommand{\GTop}{G\Top}
\newcommand{\Sset}{\mathsf{sSet}}

\newcommand{\sM}{\mathsf{M}}

\newcommand\DD{{\mathcal D}}

\newcommand\FF{{\mathcal F}}

\newcommand\LL{{\mathcal L}}
\newcommand\MM{{\mathcal M}}

\newcommand\PP{{\mathcal P}}

\newcommand\PMF{{\PP\kern-2pt\MM\FF}}

\newcommand\PML{{\PP\kern-2pt\MM\LL}}

\newcommand{\fsubd}{\mathrel{{\scriptstyle\searrow}\kern-1ex^d\kern0.5ex}}
\newcommand{\bsubd}{\mathrel{{\scriptstyle\swarrow}\kern-1.6ex^d\kern0.8ex}}
\newcommand{\fsubeq}{\mathrel{\raise-.7ex\hbox{$\overset{\searrow}{=}$}}}
\newcommand{\bsubeq}{\mathrel{\raise-.7ex\hbox{$\overset{\swarrow}{=}$}}}

\newcommand{\tsh}[1]{\left\{\kern-.9ex\left\{#1\right\}\kern-.9ex\right\}}

\newcommand{\obj}{\mathrm{Obj}}

\makeatletter
\@namedef{subjclassname@2020}{%
  \textup{2020} Mathematics Subject Classification}
\makeatother

\title{Good Objects in the Equivariant World}
\author{Surojit Ghosh}
\address{Department of Mathematics, Indian Institute of Technology, Roorkee, Haridwar Road, 247667,India}
\email{surojitghosh89@gmail.com, surojit.ghosh@ma.iitr.ac.in}
\author{Bikramjit Kundu}
\address{Department of Mathematics, Indian Institute of Technology, Roorkee, Haridwar Road, 247667,India}
\email{bikramju@gmail.com, bikramjit.pd@ma.iitr.ac.in}


\begin{document}
\begin{abstract}
This article explores equivariant localization in the category of $G$-spaces, where $G$ is a compact Lie group. We establish a commutation rule for the localization functor and the equivariant loop functor. Additionally, we introduce and classify certain good objects in this category up to their Bredon cohomology with coefficients in the constant rational Mackey functor $\underline{\Q}$.
\end{abstract}

\maketitle

\section{Introduction}
The homotopical localization functor \cite{Bo77, Fa96, Hi03} exists with respect to a class of maps between cofibrant spaces in the category of topological spaces or simplicial spaces, which are cofibrantly generated under the classical model structure. Specifically:

\begin{defn}
A space $X$ in these categories is said to be $f$-local if $X$ is fibrant and the map induces weak homotopy equivalences:
\[
\Map(B, X) \simeq \Map(A, X),
\]
where $f: A \to B$ is a map between cofibrant objects.
\end{defn}

The continuous version of the small object argument \cite{Ho99} ensures the existence of the localization functor. We have the following theorem:

\begin{theorem}\cite[Theorem A.3]{Fa96}
For any map $f: A \to B$ between cofibrant objects, there exists a localization functor $L_f$ that is coaugmented and homotopically idempotent. The map $X \to L_f(X)$ is universal with respect to $f$-local spaces.
\end{theorem}

The localisation functor commutes with the product and behaves well with mapping objects $\Map(X,T)$ with $T$ being $f$-local and $X$ being cofibrant object. In the later case the localisation functor takes the mapping object to a weakly equivalent $f$-local object. The based loop spaces being the most common as well as important mapping object naturally becomes of particular interest in this context. It turns out localisation functor $L_f$ takes the based loop object $\Omega X$ to a local object irrespective choices of $X$. Using `Segal loop machine'\cite{Se72} one can prove the following:
\begin{thm}\cite[\S 3.A]{Fa96}
For any space $X$ and map $f: A \to B$,
\[
L_f \Omega X \simeq \Omega L_{\Sigma f} X.
\]
\end{thm}

This theorem establishes a certain commutation like rule for $L_f$ and the loop functor $\Omega$. A natural question arises here. Does this rule hold in the category of $G$-topological spaces for a compact Lie group $G$? The existence of the equivariant localization functor $L_f^G$ [Theorem \ref{Glo}] is guaranteed by Chorny's work \cite{Ch05}. Our first result answers this question affirmatively.

\begin{theorem}\label{CL}
Let $X$ be a $G$-space. Then for any $G$-map $f: A \to B$,
\[
L_f^G \Omega X \simeq_G \Omega L_{\Sigma f}^G X.
\]
\end{theorem}

We also investigate the behavior of based mapping objects under the localization functor. We introduce the following definition:

\begin{defn}\label{GO}
A connected, based CW-complex $A$ is called $L$-good if for any map $f$ and space $X$ in $\Top_*$,
\[
L_f \Map_*(A, X) = \Map_*(A, Y)
\]
for some $Y \in \Top_*$.
\end{defn}

The characterization of $L$-good spaces was posed by Farjoun \cite[\S 9.f]{Fa96} and was settled by Badzioch and Dorabia\l a \cite{BaDo10} up to rational homotopy. We extend this notion to the equivariant setting:

\begin{defn}
A connected, based $G$-CW complex $A$ is $L^G$-good if for any based $G$-map $f$ and $X \in \GTop_*$,
\[
L_f^G \Map_*(A, X) \simeq_G \Map_*(A, Y)
\]
for some $G$-simply connected $Y \in \GTop_*$.
\end{defn}

In the non-equivariant case, all the k-dimensional spheres $S^k$ for $k \geq 1$ are $L$-good. This follows from Theorem \ref{CL}. The functor $L_f$ also preserves finite products up to weak equivalence, making all finite wedges of spheres $\bigvee_l S^k$ for $k \geq 1$, $l \geq 0$, $L$-good objects. Moreover, in \cite{BaDo10}, the authors showed that:

\begin{thm}\cite[Corollary 1.2]{BaDo10}
If $A$ is a simply connected, finite $L$-good space, then $A$ has the rational homotopy type of $\bigvee_l S^k$ for some $k \geq 1$, $l \geq 0$.
\end{thm}

In this article, we classify $L^{C_{p^n}}$-good $C_{p^n}$-CW complexes up to their Bredon cohomology. We have the following theorem, which is a direct consequence of Proposition \ref{chl}.

\begin{thm}
If a finite, connected $C_{p^n}$-CW complex $A$ is $L^{C_{p^n}}$-good, then $H^r_{C_{p^n}}(A;\underline{\Q}) = 0$ for all $r \notin \{0, k\}$ for some $k > 0$, where $\underline{\Q}$ is the constant rational Mackey functor.
\end{thm}

\begin{rmk}
We restrict our attention to $C_{p^n}$-CW complexes because our method relies on \cite[Theorem 1.2]{Tr83}.
\end{rmk}

\begin{notation}
\begin{itemize}\label{not}

\item $\Top$ (respectively, $\Top_{*}$ denotes the category of topological spaces (respectively, based topological spaces) and morphisms are continuous maps which will be denoted by $\Map(X,Y)$ (respectively, based continuous maps by $\Map_{*}(X, Y)$ ). The simplicial mapping complex between two objects $X, Y$ will be denoted by $\mbox{Hom}(X,Y)$ with $\Map(X\times |\Delta|^n, Y)$ as $n$-simplices.\\

\item $\Top^G$ (respectively, $\Top_{*}^G$) denotes the category of (based) $G$-spaces and for $X, Y \in \obj(\Top^G)$ (respectively,  $\obj(\Top_\ast^G)$) the morphism set $\Map (X,Y)$ (respectively, $\Map_\ast(X,Y)$) is the set of (based) continuous maps.  This category is enriched over itself where $G$ acts on morphism sets via conjugation. The simplicial enrichment is given by $\mbox{Hom}(X,Y)_n=\Map(X\times |\Delta|^n, Y)$ where $|\Delta|^n$ denotes the geometric realization of the standard simplicial complex.\\

\item $\GTop$ (respectively, $\GTop_{*}$) denotes the category of topological spaces with $G$-action (respectively, based topological spaces) and morphism sets are $G$-equivariant maps which will be denoted by $\mbox{GMap}(X,Y)$. Note that $\GTop$ (respectively $\GTop_*$) is subcategory of $\Top^G$ (respectively $\Top^G_*$). The simplicial mapping complex between two objects $X, Y$ will be denoted by $\mbox{GHom}(X,Y)_n=\mbox{GMap}(X\times |\Delta|^n, Y)$ where the action on the geometric realization of the standard simplicial complex is trivial. We will denote $G$-homotopy equivalence by $\simeq_G$.\\

\item  $C\Top_*^G$ denote the full sub category of $\Top^G_*$ whose objects are non-degenerate based $G$-spaces having the $G$-homotopy type of based $G$-CW complexes.
    
\end{itemize}

\end{notation}
\begin{ack}
The second author wish to thank David White and Ang\'{e}lica M. Osorno for some helpful conversations. The second author also wants to thank Isaac Newton Institute for Mathematical Sciences, Cambridge, for support and hospitality during the programme Topology, representation theory and higher structures, when the work on this paper was undergoing.
\end{ack}

\section{Equivariant Localization and Equivariant \texorpdfstring{$\Gamma$}{Gamma}-Spaces}

\subsection{Equivariant Localization}
Throughout the rest of the article, $G$ will denote a compact Lie group. Chorny constructed the localization functor in the equivariant category of diagrams \cite{Ch05}, which are not necessarily cofibrantly generated. We first recall some basic definitions associated with it. After that, we will apply a generalized version of Elmendorf's theorem to translate the localization functor to the category of based $G$-spaces.

Let $S$ denote $\Top$ or $\Top_*$, and let $D$ denote a small category enriched over $S$. Then $S^D$, the category of $D$-shaped diagrams over the spaces $S$, is a simplicial category whose objects are continuous functors from $D$ to $S$, and whose morphisms are simplicial sets defined by $\Hom(X, Y)_n$.
The category $S^D$ has a projective model structure, meaning weak equivalences and fibrations are defined objectwise, and cofibrations are maps that have the left lifting property with respect to trivial fibrations.

\begin{defn}\label{or}\cite{DK83}
A set of objects $\{O_e\}_{e \in E}$ in an arbitrary category $M$ is called orbits for $M$ if they satisfy the following axioms:
\begin{itemize}
\item[A0.] $M$ is closed under arbitrary direct limits.
\item[A1.] For every $O_e$, the functor $\Map(O_e, *): M \to \Sset$ commutes up to homotopy with the pushout diagram:
\[
\xymatrix{
O_{e^{\prime}} \otimes K \ar[r] \ar[d] & X_a \ar[d] \\
O_{e^{\prime}} \otimes L \ar[r] & X_{a+1}
},
\]
where $K \hookrightarrow L$ is the inclusion of finite simplicial sets.
\item[A2.] For every $O_e$, the functor $\Map(O_e, *): M \to \Sset$ commutes up to homotopy with transfinite compositions of maps $X_a \to X_{a+1}$ as in A1.
\item[A3.] There is a limit ordinal $\lambda$ such that for every $O_e$, the functor $\Map(O_e, *)$ strictly commutes with $\lambda$-transfinite compositions of maps $X_a \to X_{a+1}$ as in A1.
\end{itemize}
\end{defn}

\begin{thm}\cite{DK83}
Let $M$ be a simplicial category satisfying axioms MO and SMO of Quillen \cite{Qu67}, and let $\{O_e\}_{e \in E}$ be a set of orbits for $M$. Then $M$ admits a closed simplicial model category structure in which the simplicial structure is the given one, and a map $X \to Y$ in $M$ is a weak equivalence or a fibration if and only if, for every object $O_e$, the induced map $\Map(O_e, X) \to \Map(O_e, Y)$ is so. The cofibrations are the maps that have the left lifting property with respect to trivial fibrations.
\end{thm}

{\begin{exam}\textbf{Simplicial model structure in $\GTop_*$:}
If we consider $\GTop_*$, the set $\{G/G_{\alpha}\}_{\alpha \in A}$ for $G_{\alpha} \subset G$ satisfies the axioms in \ref{or} and becomes the set of orbits in $\GTop_*$. Here, $\GTop_*$ has a closed simplicial model structure with the set of homomorphisms between two objects $X$ and $Y$ defined by $\mbox{GHom}(X, Y)$ [See Notation \eqref{not}]. The model structure is the obvious one, with weak equivalences being fixed-point-wise and fibrations being Serre fibrations in fixed-point spaces. Cofibrations are defined by the maps having the left lifting property with respect to trivial fibrations. Similarly, $\GTop_*$ also has a closed simplicial model structure with $G_{+}$ (the disjoint union of $G$ and an added basepoint) instead of $G$ as the based $G$-space.
\end{exam}}

Now we state three theorems that will essentially give our desired localization functor in $\GTop$. Let us denote the orbit category of $G$ by $\mathcal{O}_G$, and the category of contravariant functors from $\mathcal{O}_G$ to $\Top$ by $\mbox{Fun}(\mathcal{O}^{op}_G, \Top)$. The morphisms in $\mbox{Fun}(\mathcal{O}^{op}_G, \Top)$ are natural transformations between two functors. The model structure in $\mbox{Fun}(\mathcal{O}^{op}_G, \Top)$ is again defined in the obvious way.

\begin{thm}\cite{El83}
There exists a Quillen pair of functors $\Phi$ and $\Psi$ between:
$$\mbox{Fun}(\mathcal{O}^{op}_G,\Top)\stackrel[\Phi]{\Psi}{\leftrightarrows} \GTop.$$
Moreover, the above two categories are Quillen equivalent, which makes their homotopy categories equivalent.
\end{thm}

From a slightly different perspective, this can be generalized as the following theorem. Let us denote the category of simplicial sets by $\Sset$.

\begin{thm}\cite{DK83}\label{Lo1}
Let $\sM$ be a simplicial category satisfying axioms MO and SMO of Quillen \cite{Qu67}. Let $\{O_e\}_{e \in E}$ be a set of orbits for $M$, and let $O \subset \sM$ be the full simplicial subcategory. Then the ``singular functor":
\[
\Map(O, *): \mathsf{M} \to \Sset^{O^{op}}
\]
has a left adjoint, the realization functor:
\[
O \otimes *: \Sset^{O^{op}} \to \sM.
\]
Moreover, in the model category structure, the two functors form a Quillen pair, and the two categories become Quillen equivalent.
\end{thm}

Now we return to the category of $D$-shaped diagrams and recall the definition of local objects and the existence of a coaugmented localization functor in $S^D$.

\begin{defn}\cite{Ch05}
Let $f: \underline{A} \to \underline{B}$ be a map between cofibrant objects. A diagram $\underline{X}$ is called $f$-local if $\underline{X}$ is fibrant and the induced map:
\[
\Hom(\underline{B}, \underline{X}) \simeq \Hom(\underline{A}, \underline{X})
\]
is a weak equivalence in simplicial homotopy.
\end{defn}

\begin{thm}\cite{Ch05}\label{lo2}
With $f$ as above, in the category $S^D$, there exists a coaugmented localization functor $\underline{L}_f$ that is universal in the sense that any map $\underline{X} \to \underline{Y}$ to some $f$-local diagram $\underline{Y}$ admits a unique factorization upto simplicial homotopy:
\[
\underline{X} \to \underline{L}_f \underline{X} \to \underline{Y}.
\]
\end{thm}

{We will now give the definition of $f$-local objects in $\GTop$ with respect to a $G$-map $f: A \to B \in \GTop$}. We will consider only non-based spaces here to avoid complexity. The theory is exactly analogous for based $G$-spaces.
{\begin{defn} (a) Let $f: A \to B$ is a morphism between cofibrant objects in $\GTop$.
A space $X$ in $\GTop$ is $f$-local if $X$ is fibrant and the map induces weak equivalences:\\ 
\[
\Map(B, X) \simeq \Map(A, X).
\]
This in turn gives $\Map_G(f \times G/H, X)$ is a weak equivalences in $\Top$ for each closed subgroups $H$ of $G.$
\end{defn}
\begin{prop}
Let $f$ be above in $\GTop$.
A space $X$ in $\GTop$ is said to be $f$-local iff $X$ is fibrant and the map induces weak equivalences simplicially\\ 
\[
\Hom(B, X) \simeq \Hom(A, X).
\]
\end{prop}}
\begin{proof}
Note that the adjunction gives us the isomorphism of simplicial set \[\Hom(B,X)_n\cong \mbox{Sing}\Map (B,X)_n.\] The Quillen equivalence of the adjoint pair, the geometrical realization functor and the Singular functor between model categories $|-|:\Sset\leftrightarrows \Top: \mbox{Sing}$, implies the counit map $|\mbox{Sing}(\Map(B,X))|\to \Map(B,X)$ is weak equivalence. So we have the following diagram of weak equivalence,
\begin{align*}
\Map(B,X)\simeq |\mbox{Sing}(\Map(B,X))| \simeq  |\Hom(B,X)|
\end{align*} which completes the proof.
\end{proof}
\begin{thm}\label{Glo}
In the category $\GTop$, consider a cofibration $f$ between cofibrant objects $A \to B$. There exists a coaugmented localization functor $L^G_f: \GTop \to \GTop$ that is universal in the sense that any map between $G$-spaces $X \to Y$ to some $f$-local $G$-space $Y$ admits a unique factorization upto $G$-homotopy:
\[
X \xrightarrow{j_X} L^G_f X \to Y.
\]
\end{thm}

\begin{proof}
Replace the category $M$ by $\GTop$ in Theorem \ref{Lo1}. Then the proof is a direct consequence of Theorems \ref{Lo1} and \ref{lo2}.
\end{proof}

We list some properties of the functor $L_f^G$, which are easy consequences of the universal property.

\begin{prop}
Let $f: A \to B$ be a $G$-map between cofibrant objects in $\GTop$. Then the following holds:
\begin{enumerate}
    \item For $G$-spaces $X$ and $Y$, $L_f^G(X \times Y) \simeq_G L_f^G X \times L_f^G Y$.
    \item A $G$-connected space $X$ is local with respect to $\Sigma f: \Sigma A \to \Sigma B$ if and only if $\Omega X$ is $f$-local. Here, the suspension coordinate has trivial $G$-action.
    \item If $X$ is $f$-local, then it is also $\Sigma^n f$-local for all $n \geq 0$.
\end{enumerate}
\end{prop}
\begin{proof}
We will only prove the part (1) which is exactly analogous to the non-equivariant case. First note that the product of two $f$-local spaces is $f$-local. This follows from the fact that if $T$ is $f$-local, $\mbox{Map}(A,T)$ is $f$-local for any $A$ in $\GTop$. The map $j_X\times j_Y:~X\times Y\to L_f^GX\times L_f^GY$ will factor through $L^G_f(X\times Y)$ giving us a desired map $l:~L_f^G(X \times Y) \to L_f^G X \times L_f^G Y$. To get the opposite direction map note that the adjoint of $j_{X\times Y}$ will give a map $X\to \mbox{Map}(Y, L_f^G(X\times Y))$ which will further factor through $L_f^GX$ giving us a map $L_f^GX\to \mbox{Map}(Y, L_f^G(X\times Y))$. The adjoint of this map will give us the desired map $r:~L_f^GX\times L_f^GY\to L_f^G(X\times Y)$. These two maps $l, r$ will fit together in the following $G$-homotopic commutative diagram:
\[
\xymatrix@C=3cm{
X\times Y\ar[d]^{j_{X\times Y}}\ar@/_3.0pc/[dd]_{j_X\times j_Y}\ar[rd]^{j_{X\times Y}}\\
L_f^G(X\times Y)\ar[r]^-{r\circ l}\ar[d]^{l} & L_f^G(X\times Y)\\
L_f^GX\times L_f^GY \ar[ur]^{r}.
}
\]
By universality and divisibility \cite{Ch05}, $l$ and $r$ are $G$-homotopic inverse to each other. 
\end{proof}

\subsection{Equivariant $\Gamma$-Spaces}
We recall some basic facts about equivariant $\Gamma_G$-spaces \cite{Sh89, Sa11}. Let $\Gamma_G$ denote the full subcategory of $C\Top^G_*$ of finite $G$-sets whose underlying sets are denoted by $\textbf{n} := \{0, 1, \dots, n\}$, based at $0$. The $\Gamma_G$-spaces are defined to be the $G$-equivariant functors from $\Gamma_G \to \GTop_*$. It is to be noted that $\Gamma_e$ is essentially Segal's $\Gamma$-space with $\Gamma^{\operatorname{op}}$.

\begin{defn}\label{GamGs}
A $\Gamma_G$-space $\mbox{A}$ is called \emph{special} if:
\begin{enumerate}
    \item $\mbox{A}(\textbf{0})$ is $G$-contractible.
    \item The adjoint map $P_\textbf{n}: \mbox{A}(\textbf{n}) \to \Map_*(\textbf{n}, \mbox{A}(\textbf{1}))$ of the based $G$-map $\textbf{n} \wedge \mbox{A}(\textbf{n}) \to \mbox{A}(\textbf{1})$ induced by the map $p_j$ is a $G$-homotopy equivalence. Here, $p_j(\textbf{n} - \{j\}) = 0$ and $p_j(j) = 1$.
\end{enumerate}
\end{defn}

Given a $\Gamma_G$-space $\mbox{A}$, we can form a new $G$-space defined by $\mbox{S}_G \mbox{A}_V$:
\[
\mbox{S}_G \mbox{A}_V = \mbox{B}(S^V, \Gamma_G, \mbox{A}) / \mbox{B}(\infty, \Gamma_G, \mbox{A}),
\]
where $\mbox{B}(M, \Gamma_G, \mbox{A})$ is the classifying space of the usual two-sided Bar construction $\mbox{B}_*(-, \Gamma_G, \mbox{A})$, and $S^V$ is the contravariant $G$-functor taking $S \mapsto \Map_*(S, S^V)$. It is to be noted that the $G$-action on $S_G \mbox{A}_V$ is induced from the $G$-simplicial structure of the Bar construction. The association $\mbox{A}(-) \mapsto \mbox{B}(-, \Gamma_G, \mbox{A})$ is functorial on the category of $\Gamma_G$-spaces.

Shimakawa proved the following theorem, which is essentially the equivariant version of Segal's work in \cite{Se72}, stating that every special $\Gamma$-space gives rise to an $\Omega$-spectrum.

\begin{thm}\cite{Sh89}\label{eqsem}
Every special $\Gamma_G$-space gives rise to an almost $\Omega$-$G$-spectrum graded over finite-dimensional $G$-representations $\mbox{V}$, given by $\mbox{S}_G \mbox{A}_V$. The structure map:
\[
\epsilon_{V,W}: S_G \mbox{A}_W \to \Omega^V S_G \mbox{A}_{V \oplus W}
\]
is a $G$-homotopy equivalence whenever $W^G \neq \emptyset$. In particular, $\epsilon: S_G \mbox{A}_0 \to \Omega S_G \mbox{A}_{\R}$ is a $G$-homotopy equivalence if and only if $\mbox{A}(\textbf{1})^H$ is grouplike for every subgroup $H \subset G$.
\end{thm}

\subsection{Equivariant Cohomology}
Let us begin by defining $G$-CW complexes in the category $\GTop$. A $G$-CW complex is formed by attaching cells of the type $G/H \times \DD^n$ in increasing dimensions. Note that a $G$-action on a disjoint union of points is classified up to $G$-equivalence by the decomposition into orbits, and therefore, the $0$-cells are a disjoint union of copies of $G/H$ for conjugacy classes of subgroups $H$. Equivariant $n$-cells are the $G$-spaces $G/H \times D^n$ (with trivial action on $D^n$), where $H \subseteq G$ runs over conjugacy classes of subgroups. The boundary of such a cell is $G/H \times S^{n-1}$. One defines:

\begin{defn}
A relative $G$-CW complex is a pair of $G$-spaces $(X, A)$ together with a filtration $\{X^{(n)}\}$ of $X$ such that:
\begin{itemize}
    \item[(a)] $X^{(0)} = A \bigcup \amalg_{\alpha} G/H_{\alpha}$.
    \item[(b)] $X^{(n+1)}$ is obtained from $X^{(n)}$ as a pushout:
    \[
    \xymatrix{
    \amalg_{j \in J} G/H_j \times S^n \ar[d] \ar[r] & X^{(n)} \ar[d] \\
    \amalg_{j \in J} G/H_j \times D^{n+1} \ar[r] & X^{(n+1)}
    }.
    \]
    \item[(c)] $X = \cup_{n \in \mathbb{Z}} X^{(n)}$ has the colimit topology.
\end{itemize}
\end{defn}

A coefficient system $\underline{M}$ is a contravariant functor from $\mathcal{O}_G$ to the category of Abelian groups. Given a contravariant functor from $\mathcal{O}_G$ to spaces, taking chains yields such a coefficient system of chain complexes. A particular example is the singular chain complex functor $\underline{C}_*(X)$, given by $G/H \mapsto C_*(X^H)$ for any $G$-space $X$.

\begin{defn}
Let $\underline{M}$ be a coefficient system. Define:
\[
C^*_G(X; \underline{M}) := \Hom_{\mathcal{O}_G}(\underline{C}_*(X), \underline{M})
\]
as the singular cochain complex with coefficients in $\underline{M}$. The cohomology of this complex is denoted by $H^*_G(X; \underline{M})$ and is called the Bredon cohomology of $X$ with coefficients in $\underline{M}$.
\end{defn}

\begin{rmk}
In the rest of the section, by Mackey functor, we will mean the underlying coefficient system.
\end{rmk}

\section{Good Objects}\label{GO}
We have introduced the notion of good objects in Introduction \ref{GO}. As our first example, we show that $G$-loop spaces are, in fact, good objects in $\GTop$.

\begin{prop}\label{CAL}
Let $X$ be a $G$-connected space. Then the based loop space $\Omega X$ has the associated $G$-action. The coaugmentation map $\Omega X \to L_f^G \Omega X$ is $G$-homotopic to a $G$-loop map, meaning we can find $G$-spaces $A, B$ such that the following diagram commutes, with the vertical maps being $G$-homotopy equivalences:
\[
\xymatrix{
\Omega X \ar[r] \ar[d]^{\simeq_G} & L_f^G \Omega X \ar[d]^{\simeq_G} \\
\Omega A \ar[r] & \Omega B.
}
\]
\end{prop}

\begin{proof}
The proof is essentially analogous to the non-equivariant case \cite{Fa96}. The main tool is Shimakawa's version of equivariant Segal's machinery to detect a loop space \cite{Se72, Sh89}. Observe that $(\Omega X)^H = \Omega(X^H)$ for every subgroup $H \leq G$. Define a $\Gamma_G$-space $\mbox{A}$ by:
\begin{align*}
&A(\textbf{0})=~~*\\ 
&\mbox{A}(\textbf{n})=(\Omega X)^n
\end{align*}
where the $G$-action on $(\Omega X)^n$ is coming from the action on the finite $G$-set $\textbf{n}$.
Since the adjoint map $P_n$ in Definition \ref{GamGs} is essentially given by the identity map $\Omega X \to \Omega X$, our $\Gamma_G$-space $\mbox{A}$ is special. Deploying the ``equivariant loop machine" in Theorem \ref{eqsem}, we have:
\[
\Omega X \simeq_G S_G \mbox{A}_0 \simeq_G \Omega S_G \mbox{A}_{\R}.
\]
Now applying the $L^G_f$ functor to the $\Gamma_G$-space $\mbox{A}$ will produce another special $\Gamma_G$-space $L^G_f \mbox{A}$. This is justified by the fact that:
\[
L_f^G(X \times Y) \simeq_G L_f^G X \times L_f^G Y.
\]
Now all we need is to find the ``group-like" structure in $L_f^G \mbox{A}(\textbf{1})$ to apply the ``equivariant loop machine" on $L_f^G \mbox{A}(\textbf{n})$. This follows immediately from the naturality of the $G$-homotopy equivalence:
\[
L_f^G(X \times X) \simeq_G L_f^G X \times L_f^G X,
\]
and from the fact that any $L_f^G$ will take connected components to connected components.

This completes the proof by giving the desired diagram:
\[
\xymatrix{
\Omega X \ar[r] \ar[d]^{\simeq_G} & L_f^G \Omega X \ar[d]^{\simeq_G} \\
\Omega S_G \mbox{A}_{\R} \ar[r] & \Omega S_G L_f^G \mbox{A}_{\R}.
}
\]
\end{proof}

It is to be observed that the above proof does not require a monoidal structure in $\mbox{A}(\textbf{1})$, unlike the non-equivariant case described in \cite{Fa96}, for its direct passage to the functorial Bar construction.

We will denote by $\bar{B}$ the delooping functor. For any $Y \in \text{Obj}(\GTop)$, $\bar{B} Y$ is the unique object in $\GTop$ with $Y \simeq_G \Omega(\bar{B} Y)$. In general, the delooping functor does not always exist.

\begin{theorem}
Let $X$ be a $G$-space. Then for any $G$-map $f: A \to B$:
\[
L_f^G \Omega X \simeq_G \Omega L_{\Sigma f}^G X.
\]
\end{theorem}

\begin{proof}
The proof is exactly analogous to the non-equivariant case as in \cite[\S A.4]{Fa96}. We need to produce a $G$-homotopy between $L_f^G \Omega X$ and $\Omega L_{\Sigma f}^G X$. Since $\Omega L_{\Sigma f}^G X$ is an $f$-local space by the universal property of $L_f^G$, we have a map:
\[
i: L_f^G \Omega X \to \Omega L_{\Sigma f}^G X.
\]
By Proposition \ref{CAL}, $\bar{B}(L_f^G \Omega X) := S_G L_f^G \mbox{A}_{\R}$. This gives a map:
\[
X \to \bar{B}(L_f^G \Omega X).
\]
Since $\bar{B}(L_f^G \Omega X)$ is $\Sigma f$-local, we have the following universal diagram:
\[
\xymatrix{
X \ar[r] \ar[d] & \bar{B}(L_f^G \Omega X) \\
L_{\Sigma f}^G X \ar@{-->}[ur].
}
\]
Applying the loop functor to the universal map, we have our desired map in the opposite direction:
\[
j: \Omega L_{\Sigma f}^G X \to L_f^G \Omega X.
\]
By the universal property, both compositions $i \circ j$ and $j \circ i$ are $G$-homotopic to the identity.
\end{proof}

\begin{theorem}\label{mapeilen}
Let $A$ be a finite $C_{p^n}$-CW complex. Suppose the integer-graded Bredon cohomology of $A$ vanishes above degree $r$. For $m > i > r$, we have:
\begin{equation}\label{eilen1}
\text{Map}_*(A, K(\underline{\Q}, m)) \simeq_{C_{p^n}} \prod_{i=m-r}^{m} K(\underline{M}^{m-i}, i),
\end{equation}
where $\underline{\Q}$ is the constant Mackey functor taking every $C_{p^n}$-set to $\Q$, and $\underline{M}^i$ is the $C_{p^n}$-coefficient system defined on by:
\[
\underline{M}^i(C_{p^n}/P) = H_P^{m-i}(A; \underline{\Q}).
\]
\end{theorem}

The proof will follow from the following lemma:

\begin{lemma}\cite{HHR21}
For any Mackey functor $\underline{M}$ and each integer $n \geq 0$, there is a $G$-space $K(\underline{M}, n)$, the equivariant Eilenberg-Mac Lane space, such that:
\[
\pi^H_k K(\underline{M}, n) = 
\begin{cases}
\underline{M}(G/H) & \text{for } k = n, \\
0 & \text{otherwise}.
\end{cases}
\]
\end{lemma}

\begin{proof}[Proof of Theorem \ref{mapeilen}]
It is known that the rationalization of every Hopf $G$-space is a product of equivariant Eilenberg-Mac Lane spaces up to $G$-homotopy \cite{Tr83}. As:
\[
\text{Map}_*(A, K(\underline{\Q}, m)) \simeq \Omega \text{Map}_*(A, K(\underline{\Q}, m+1)),
\]
the left-hand side of \eqref{eilen1} in Theorem \ref{mapeilen} is an H-space. It is also a rational space since we are considering the constant rational Mackey functor $\underline{\Q}$. Thus, we have:
\[
\text{Map}_*(A, K(\underline{\Q}, m)) \simeq_{C_{p^n}} \prod_{i=0}^N K(\underline{M}^{i_*}, i).
\]
Now, applying the functor $\pi^P_i$ to both sides for any subgroup $P \subset C_{p^n}$, we have:
\[
\underline{M}^{i_*}(C_{p^n}/P) = [S^i, \text{Map}_*(A, K(\underline{\Q}, m))^P] = [A, \Omega^i K(\underline{\Q}, m)]^P = H_P^{m-i}(A, \underline{\Q}).
\]
Rearranging the indices and replacing $i_*$ with $m - i$, we obtain the desired result.
\end{proof}

\begin{prop}\label{chl}
Let $A$ be a finite, connected $C_{p^n}$-CW complex such that for some $r > s > 0$, $H^r_{C_{p^n}}(A; \underline{\Q}) \neq 0 \neq H^s_{C_{p^n}}(A; \underline{\Q})$. Then $A$ is not an $L^{C_{p^n}}$-good space.
\end{prop}

\begin{proof}
For $m > r + p^n$, consider the $C_{p^n}$-equivariant map $f: S^{\rho} \to *$, where $\rho = \rho_{C_{p^n}} \oplus (m - r - p^n)\epsilon$ is the $(m - r)$-dimensional $C_{p^n}$-representation. Here, $\rho_{C_{p^n}}$ denotes the $C_{p^n}$-regular representation, and $\epsilon$ denotes the one-dimensional trivial representation. Observe that:
\[
L^{C_{p^n}}_f(\text{Map}_*(A, K(\underline{\Q}, m))) 
\]
is just the $C_{p^n}$-equivariant Postnikov section $P^{C_{p^n}}_{m - r + 1} \text{Map}_*(A, K(\underline{\Q}, m))$. By Theorem \ref{mapeilen}, this implies:
\[
L^{C_{p^n}}_f \text{Map}_*(A, K(\underline{\Q}, m))) \simeq_{C_{p^n}} K(\underline{M}^r, m - r).
\]
If $A$ is an $L^G$-good space, then for every $N > 0$, we can find a $C_{p^n}$-space $Y$ such that:
\[
L^{C_{p^n}}_f \text{Map}_*(A, K(\underline{\Q}, m))) \simeq_{C_{p^n}} \text{Map}_*(A, Y) \simeq_{C_{p^n}} K(\underline{M}^r, N).
\]

Let us first assume $A$ is $L^{C_{p^n}}$-good, $C_{p^n}$-simply connected with $H^r_{C_{p^n}}(A; \underline{\Q}) \neq 0$, and for $N > r + 1$, such a $Y$ exists with $Y$ being a $C_{p^n}$-simply connected space. Since $K(\underline{M}^r, N)$ is a rational space, we can assume $Y$ is $C_{p^n}$-rational because of the equivalence:
\[
\text{Map}_*(A, Y) \simeq_{C_{p^n}} \text{Map}_*(A, Y_0),
\]
where $Y_0$ is a $C_{p^n}$-rational space. By \cite[Theorem 1.2]{Tr83}, we have:
\[
\Omega Y \simeq_{C_{p^n}} \prod_{n \geq 1} K(\underline{N}^i; i),
\]
where $\underline{N}^i$ is a suitable Mackey functor taking values in rational vector spaces. Now we have the following equivalences:
\begin{equation}\label{4equiv}
\Omega \text{Map}_*(A, Y) \simeq_{C_{p^n}} K(\underline{M}^r; N - 1) \simeq_{C_{p^n}} \text{Map}_*(A, \Omega Y) \simeq_{C_{p^n}} \text{Map}_*(A, \prod_{n \geq 1} K(\underline{N}^i; i)).
\end{equation}
There should be $n_0 \geq N - 1$ such that $\underline{N}^{n_0}$ is a non-zero Mackey functor. Otherwise, if all $\underline{N}^i$ are zero for $i \geq N - 1$, we have the following $C_{p^n}$-homeomorphism:
\[
\text{Map}_*(A, \prod_{i} K(\underline{N}^i; i)) = \prod_{i=1}^{N - 2} \text{Map}_*(A, K(\underline{N}^i; i)).
\]
This implies:
\[
\pi^P_i \text{Map}_*(A, \Omega Y) = \oplus_{i=1}^{N - 2} H^{N - i}_{C_{p^n}}(A; \underline{N}^i),
\]
which contradicts the first equivalence of \eqref{4equiv}. More specifically, $\Omega \text{Map}_*(A, Y)$ has two non-trivial groups in dimensions $n_0 - p$ and $n_0 - q$ as:
\[
\pi^P_{n_0 - p} \text{Map}_*(A, \Omega Y) = H^p_{C_{p^n}}(A, \underline{N}^{n_0}) \neq 0,
\]
and
\[
\pi^P_{n_0 - q} \text{Map}_*(A, \Omega Y) \simeq_{C_{p^n}} H^q_{C_{p^n}}(A, \underline{N}^{n_0}) \neq 0.
\]
This again contradicts the first equivalence of \eqref{4equiv} and completes the proof.
\end{proof}

\end{document}